\documentclass[10pt]{article}

\usepackage{amsmath,amssymb,amsfonts,amsthm,amscd,eucal,amssymb}
\usepackage[english]{babel}

\newtheorem{thm}{Theorem}[section]

\newtheorem{lem}[thm]{Lemma}

\newtheorem{hypothesis}[thm]{Hypothesis}

\newtheorem{prop}[thm]{Proposition}

\theoremstyle{definition}
\newtheorem{defn}[thm]{Definition}

\newtheorem{rem}[thm]{Remark}

\def\bE{\mathbb{E}}

\def\cF{\mathcal{F}}

\def\PP{\mathbb{P}}
\def\bA{\mathbf{A}}
\def\la{\langle}
\def\ra{\rangle}
\def\ldueo{L^2(\mathcal{O})}
\def\hunozero{H^1_0(\mathcal{O})}
\def\dd{{\rm d}}

\def\bW{\mathbf{W}}
\def\bJ{\mathbb{J}}

\def\cL{\mathcal{L}}

\def\tF{\tilde{\mathcal{F}}}

\def\tX{\tilde{X}}
\def\btW{\tilde{\mathbf{W}}}
\def\tY{\tilde{Y}}
\def\tW{\tilde{W}}
\def\tZ{\tilde{Z}}
\def\cU{\mathcal{U}}
\def\R{\mathbb{R}}
\def\bN{\mathbb{N}}
\def\bU{\mathbb{U}}

\def\cH{\mathcal{H}}
\def\cF{\mathcal{F}}
\def\bF{\mathbf{F}}
\def\cQ{\mathbf{Q}}
\def\bQ{\mathbf{Q}}
\def\cO{\mathcal{O}}
\def\cX{\mathcal{X}}

\def\ep{\varepsilon}
\begin{document}

  \title{Optimal control for stochastic heat equation with memory}
  \def\lhead{F.\ CCONFORTOLA, E.\ MASTROGIACOMO} 
  \def\rhead{Optimal control for stochastic Volterra equations} 

\author{
  Fulvia CONFORTOLA\thanks{Email address: fulvia.confortola@polimi.it}
  \\
  Elisa MASTROGIACOMO\thanks{Email address: elisa.mastrogiacomo@polimi.it}
}

\maketitle

  \begin{abstract}
    In this paper, we investigate the existence and uniqueness of solutions for a class
    of evolutionary integral equations perturbed by a noise arising in the theory of heat conduction. As a motivation of our results, we study an optimal
control problem when the control enters the system together with the noise.
  \end{abstract}

\begin{center}\begin{minipage}{.8\textwidth}{%

    {\par\leavevmode\hbox {\it 1991 MSC:\ }} 45D05{\unskip, }
    93E20{\unskip, } 60H30.  }
\end{minipage}\end{center}
\date{\null}
\section{Introduction}\label{sec:intr}

Our main goal in this paper is to analyse a class of stochastic integro-differential equation arising in 
the theory of heat conduction for materials with memory and to present an application to an optimal control problem where
the control enters the system together with the noise.
Needless to say that many physical phenomena are better described if one considers
in the equation of the model some terms which take into consideration the past hystory of the system. Further, it is sensible to assume that the modeld of certain phenomena
from the real world are more realistic if some kind of uncertainity, for instance,
some randomness or enviromental noise, is also considered in the formulation.

We wish to mention that applications to optimal control problems naturally arise in the 
study of heating processes, for example in modeling heating with radiation boundary condition, 
simplified superconductivity, control of stationary flows, glueing in polymeric materials (for a thorough introduction to these problems we refer to the standard monograph by Lions \cite{Lio71} or Fredi \cite{tro05}.


Here we are concerned with the following semilinear heat equation 
\begin{equation}\label{eq:Volt1}
   \begin{aligned}
     \partial_t v(t,x) =k_0\Delta v(t,x)+\int_{-\infty}^t k_1(t-s) \Delta v(s) \dd s +
      g(t,x,v(t,x)) 
   \end{aligned}
\end{equation}
in the bounded domain $\cO\subset \R^d$ with Dirichlet boundary condition
\begin{align}\label{eq:bc}
     v_{|\partial \cO}(t,x)=0, \quad t\in \R, x\in \partial \cO
\end{align}
and initial condition given by
\begin{align}\label{eq:ic}
     v(s,x)=v_0(s,x), \quad s\leq 0,x\in \cO.
\end{align}
Notice that $v_0(\cdot)$ represents the \emph{past history} of the system
and should satisfy suitable smoothness properties (as we will see later on).
Moreover, the function $k(t)=k_0+\int_0^t k_1(s) \dd s$ is called the 
\emph{convolution kernel} of the system and $k_1$ is assumed to be 
$3$-monotone (see Hypothesis \ref{hp:kernel} for the precise definition of this term).

We are interested in the analysis of the system \eqref{eq:Volt1} when
the function $g$: \begin{enumerate}
\item[$i.$] is given by a Lipschitz continuous term $f$ and an additive gaussian noise $W$ with covariance $Q$, i.e.
\begin{align*}
    g(t,x,v)= f(t,v(t,x)) + \sqrt{Q}\partial_t W(t);
\end{align*}
\item[$ii.$] depends on a further parameter $\gamma$ which introduces a control process in the system;
this means that
\begin{align*}
     g=g(t,x,v,\gamma)= f(t,v(t,x)) + \sqrt{Q}(r(t,v(t,x),\gamma(t,x))+\partial_t W(t)),
\end{align*}
where $r$ is a function with appropriate regularity.
\end{enumerate}
The main question arising around the first case (which we refer to as \emph{uncontrolled problem}) is to determine 
existence and uniqueness of the solution. This problem can be handled by reducing 
equation \eqref{eq:Volt1} to an abstract Cauchy problem on an appropriate product 
space, which contains the whole history of the solution.  Within this framework the system can be represented with the following evolutionary equation
\begin{equation*}
\begin{cases}
    \dd X(t) = \bA X(t) \dd t + \bF (t,X(t))\dd t + \sqrt{\bQ}\dd W(t) \\
    X(0)=X_0,
\end{cases}
\end{equation*}
where $\bA$ is the generator of a $C_0$-semigroup, $\bF$ is a Lipschitz continuous function 
$\bQ$ a linear operator and $\bW$ a vector defined in term of the Wiener process
$(W(t))_{t\geq 0}$.
 Similar approaches are widely used
in literature, see for example Miller \cite{miller/1974} and Dafermos \cite{daf70} for the deterministic
counterpart, Caraballo and Chueshov \cite{car-chu08,car-chu07} and the more recent work
\cite{BoDaTu} for stochastic models. Anyway, differently for them, we are able to treat 
  more general kernels $k$. 

We stress that our approach may has the advantage that it naturally links the solution of a
Volterra equation to a Markov process; this has the important development in view of the application of the
analytic machinery to Volterra equations to solve the optimal control
problems.
 
In the  case $g$ contains a control parameter, the natural problem is to
  determine a solution of the Volterra equation and a control process $\gamma$, within a set of admissible controls, in such a way that they minimize a cost functional. 
In particular in this paper we consider a cost of the form:
\begin{align*}
     \bJ(v_0,\gamma)= \bE \int_0^T \int_{\bar{\cO}} \ell(t,v(t,\xi),\gamma(t,\xi) \dd \xi \ \dd t 
+ \bE \int_{\bar\cO} \phi(v(T,\xi)) \dd \xi,
\end{align*}
where $\ell$ and $\phi$ are given real functions. 

In the same way as the uncontrolled problem, the model can be translated into an abstract setting. In particular, it can be rewritten in the form
\begin{equation*}
\begin{cases}
    \dd X(t) = \bA X(t) \dd t + \bF (t,X(t))\dd t + \sqrt{\bQ}( R(t,X(t),\gamma(t)) \dd t+ \dd W(t)) \\
    X(0)=X_0,
\end{cases}
\end{equation*}
where $\bA,\bF,\bQ,\bW, X_0$ are as above and $R$ is given in terms of the function $r$.
Notice here the special structure of the control term, which is clearly a restriction; however it arises from concrete models. 
 Due to the special structure of the control term we are able to perform the synthesis of the optimal control, by solving in the weak sense 
the closed loop equation. Thus, we can characterize optimal controls by a feedback law. 

The paper is organized as follows: in the next subsection we give the physical motivation
of our work; in Section \ref{sec:genass} we introduce the main assumptions on the coefficients of the problem while in Section \ref{sec:anal} we reformulate the uncontrolled problem into a semilinear abstract evolution equation and we study the properties of leading operator, while 
in Section \ref{sec:stoc-conv} we are concerned with the stochastic convolution of the
rewritten equation. To this end,
we study the so-called resolvent family (see Subsection \ref{ssec:res-fam}) and the scalar resolvent family (see Subsection \ref{ssec:screfa}) associated with our problem. In Section \ref{sec:exun} we prove the first main result of the paper: we determine existence and uniqueness of the solution of the uncontrolled Volterra equation \ref{eq:Volt1}. Finally, in Section \ref{sec:cont} we perform the standard synthesis of the optimal control.  

\subsection{Motivation}
    Let us briefly explain one possible physical meaning of our model. Let $\cO$ be a 
  $3$-dimensional \emph{homogeneous} and \emph{isotropic} rigid body (see
Pr\"uss \cite[p.125]{pruss} for more details on the physical terminology) which is represented by
 an open set $\cO \subset \R^d$ ($d=1,2,3$) with boundary $\partial \cO$ of class $C^1$.
Points in $\cO$ (i.e. material points) will be denoted by $x, y, \dots$.
Suppose that the body $\cO$ is subject to temperature changes.
We denote by
$v=v(t,x)$ the temperature at time $t \in \R^+$, $q(t,x)$ the heat flux vector field, $e(t,x)$ the temperatue and $f(t,x,v)$ the heat
supply (possibly depending on the solution itself).

%
%

We denote by
$v=v(t,x)$ the temperature at time $t \in \R^+$, $q(t,x)$ the heat flux vector field, $e(t,x)$ the temperatue and $f(t,x)$ the external heat
supply. Balance of energy then reads as:
\begin{align}\label{eq:en-bal}
     \partial_t e(t,x)= -{\rm div} q(t,x)+ f(t,x), \quad t\in \R,x\in \cO,
\end{align}
with the boundary conditions basically either prescribed temperature or prescribed
heat flux through the boundary. In particular, one (natural) choice is represented by
Dirichlet boundary conditions:
\begin{align*}
      v(t,x)_{|x\in \partial \cO}=0, \quad t\in \R.
\end{align*}
For the relationship between $e$ and $v$ we shall use the following linear
law (or, more formally, \emph{constitutive law}):
\begin{align*}
    e(t,x)=\int_0^\infty \dd m(r) v(t-r,x) \dd r + e_\infty, \qquad t\in \R^+, x\in \cO;
\end{align*}
where $e_\infty$ is a suitable positive phenomenological constant.
Analogously, for the constitutive law relating $q$ and $v$ we choose
\begin{align*}
    q(t,x)=-\int_0^\infty \dd k(r)\nabla v(t-r,x), \qquad t\in \R, x\in \cO,
\end{align*}
where $m,k\in BV_{loc}(\R^+)$ are scalar functions.

Rearranging equation \eqref{eq:en-bal}, we arrive at the following non autonomous heat-equation
with memory
\begin{align*}
    &\dd m *\partial_t v(t,x)= \dd k *\Delta v(t,x) + f(t,x), \qquad t>0, x\in \cO;
    \\
   &v(t,x)=0, \quad t\in \R_+, x\in \partial \cO,
\end{align*}
where $*$ denotes the symbol for the convolution product between two functions.
\begin{rem}
From the literature one can infer that $m$ is a {\emph creep function}, i.e. it is nonnegative,
nondecreasing and concave which is also bounded. The natural form of this kind of functions is given by
\begin{align*}
    m(t)=m_0 + m_\infty t + \int_0^t m_1(r)   \dd r,
\end{align*}
for $m_0\geq 0$, $m_\infty \geq 0$ (in our case, $m_\infty=0$) and $m_1\in L^1(\R_+)$.

From a physical point of view, $m_0$ corresponds to the \emph{istantaneous heat capacity}, i.e.
the ratio of the change in heat energy of a unit mass of a substance to the change in temperature of the substance. The function $m_1$ is called \emph{energy-temperature relaxation function} 
while $m(\infty)=m_0+\int_0^\infty m_1(s)\dd s$ is termed \emph{equilibrium heat capacity}.
In accordance with several works concerning with same type of problems (see, for example,
Cl\'ement and Nohel \cite{cle-noh79}, Nunziato\cite{Nunziato1971}, Monnieaux and Pr\"uss  \cite{mon-pru97},  Grasselli and Pata \cite{gra-pat01,gra-pat06}), in this paper we choose for semplicity $m(t)\equiv m_0=1$.
 
Concerning the function $k$, the literature is somewhat controversial. From Gurtin and Pipkin
\cite{gur-pip68} and Nunziato \cite{Nunziato1971} one can expext that $k$ is a bounded creep function as well, in particular
$k_\infty=0$, $k_1\in L^1(\R^+)$ and $k(\infty)=k_0+\int_0^\infty k_1(s) \dd s >0$.
The constant $k_0$ is termed \emph{istantaneous conductivity}, $k(\infty)$ is called
\emph{equilibrium conductivity} while $k_1$ is called \emph{heat conduction relaxation function}.
On the other hand, Cl\'ement and Nohel \cite{cle-noh79}, Cl\'ement and Pr\"uss \cite{cle-pru90} and Lunardi
\cite{lun90} and  write $k(t)=k_0-\int_0^t k_1(s)\dd s >0$ with $k_1$ positive and nonincreasing; in this case $k$ is $2$-monotone
(see Hypothesis \ref{hp:kernel} for the explanation of this term). Also Bonaccorsi and Da Prato and Tubaro \cite{BoDaTu} consider $k$ as above but they require $k_1$ completely monotone. In this theory the equilibrium conductivity $k(\infty)=k_0-\int_0^\infty k_1(s) \dd s$ is smaller then the istantaneous conductivity, in contrast with Nunziato. 

We stress that in the present paper, we assume that $k$ has the same form as in 
Gurtin and Pipkin \cite{gur-pip68} and Nunziato \cite{Nunziato1971}.
\end{rem}

\section{General assumptions}\label{sec:genass}
In equation \eqref{eq:Volt1}, we are given the kernel $k:\ \R \to \R$, the non linear term 
$f: \ \R \to \R$ and the stochastic perturbation $(W(t))_{t\geq 0}$. 

We assume the following. 

\begin{hypothesis}\label{hp:kernel}

The kernel $\left\{k(t): t\geq 0\right\}$ is a {\emph creep function}, that is,
     $k(t):=k_0+\int_0^t k_1(s) {\rm d}s$, where $k_0>0$ and the function $k_1$ is ${\bf 3}$-{\bf monotone}, that is, it 
 satisfies the following conditions:
   \begin{enumerate}
         \item[h1)] $k_1\in L^1(\R^+)\cap C^1(\R^+)$;
         \item[h2)] $k_1$ is positive and nonincreasing;
        \item[h3)]  $-k_1^\prime$ is nonincreasing and convex;
   \end{enumerate}
\end{hypothesis}

\begin{rem}
      We stress that the above assumption allows $k_1(t)$ to have a singularity at $t=0$, whose order is less than $1$, since $k(t)$ is a non-negative function in
$L^1(\R^+)$. For instance, we are able to consider a weakly singular kernel of the following type
  \begin{align*}
       k_1(t):=\frac{e^{-\delta t}}{t^\gamma}, \quad 0\leq \gamma <1.
  \end{align*}
\end{rem}

\begin{rem}
    Following the terminology introduced by Pr\"uss \cite[Definition 4.4 pag. 94]{pruss}, the function
$k$ belongs to the class of \emph{creep functions}. As stated in Section \ref{sec:intr}, the function $k$ has a physical meaning within the theory of materials with memory.
\end{rem}

Concerning the nonlinear part of the system we have:
\begin{hypothesis}\label{hp:f}
  The function $f:\ [0,T] \times \R \to \R$ satisfies the following conditions:
  \begin{enumerate} 
  \item   $f$ is continuous and differentiable on $\R$.
   \item  $f$ is Lipschitz continuous with respect to $x$,  uniformly on $t$, and has sublinear growth; this means that there exists a constant  $L > 0$ such that
  \begin{align*}
   &|f(t,x)| \leq L(1 + |x|)   \qquad t\geq 0, x\in \R \\
 &  |f(t,x) - f(t,y)| \leq L|x - y| \qquad t\geq 0, x,y\in \R.
   \end{align*}
  \end{enumerate}
\end{hypothesis}

The conditions on the stochastic perturbation are given in the following.
\begin{hypothesis}\label{hp:noise}
     \begin{enumerate}
        \item[]
     \item The process $(W(t))_{t\geq 0}$ is a cylindrical Wiener process defined on a complete
probability space with values in $L^2(\cO)$. In particular
   $W(t)$ is of the form
    \begin{align*}
       \la W(t),x\ra = \sum_{k=0}^\infty \la e_k,x \ra \beta_k(t), \qquad t\geq 0, x\in L^2(\cO)
    \end{align*}
     where $\left\{ \beta_k\right\}_{k\in\bN}$ is a sequence of real, standard, independent
     Brownian motions on $(\Omega, \cF, (\cF_t)_{t\geq 0}, \PP)$.
      \item $Q$ is a linear bounded operator, symmetric and positive. With no loss of generality,
  we shall assume in the sequel that $A$ and $Q$ diagonalizes on the same basis of $L^2(\cO)$
   (this is required only for semplicity);
    \item \label{it:AQ} If $\left\{\mu_j\right\}_{j\in \bN}$ and $\left\{\lambda_j\right\}_{j\in \bN}$ are respectively the eigenvalues of $A$ and $Q$ then we require
    \begin{align*}
          {\rm Tr}[(-\Delta)^{1-(1+\theta)/\delta}Q] = \sum_{j=1}^\infty \frac{\lambda_j}{\mu_j^{(1+\theta)/\delta-1}} <\infty,
   \end{align*}
    where $\delta$ is the quantity 
   \begin{align}\label{eq:delta}
          \delta:=1+\frac{2}{\pi}\sup\{ |{\rm arg}\, \hat{k}(\lambda)|: {\rm Re}\lambda >0\}
    \end{align}
   and $\theta$ is any real number in  $ (0,1)$ such that $1+\theta >\delta$.
    \end{enumerate}
\end{hypothesis}

\begin{rem}
\begin{enumerate}
\item[]
\item[$i.$]
   We notice that the quantity $\delta$ introduced in \eqref{eq:delta} depends only on the 
   behaviour of the Laplace transform of the kernel $k$. In Pruss and Monnieaux \cite{mon-pru97} it is proved that, for the class of kernels considered by us (i.e. for $3$-monotone kernels), 
   the Laplace transform $\hat{k}$ satisfies the following bound: 
\begin{align}\label{eq:sect} 
\sup\{|{\rm arg}\,\hat{k}(\lambda)|: {\rm Re} \lambda>0\}=\theta <\frac{\pi}{2}
\end{align}
  and, consequently, $\delta$ belongs to $(1,2)$.  
 Following the terminology in Pr\"uss (see \cite{pruss}) we say that 
  the kernel $k$ is $\theta$-sectorial.
  
  It can be proved that the sectoriality of the kernel 
   plays a central role in the study of the Volterra equation
\ref{eq:Volt1}. In particular, it
   allows to prove existence of the resolvent family corresponding with the problem, and consequently to investigate existence and uniqueness of the solution.
   For more details we refer to Section \ref{ssec:res-fam} and the monograph  \cite[Section 3]{pruss}.

   \item[$ii.$]
   As has been observed, $\delta \in (1,2)$; then condition \ref{it:AQ} in Hypothesis
    \ref{hp:noise} 
   implies that ${\rm Tr}[QA^{-\epsilon}]$ for any $\epsilon \in [0,2/\delta-1)$.  
   We stress that this condition is authomatically satisfied if $Q$ is of trace class. 
   \end{enumerate}
\end{rem}
\section{Statement and Reformulation of the uncontrolled equation}\label{sec:anal}

\subsection{The abstract setting}
In this section we are concerned with the following (uncontrolled) class of integral Volterra equations perturbed
by an additive Wiener noise
\begin{equation}\label{eq:uncont}
   \begin{aligned}
     &\partial_t v(t,x)= k_0 \Delta v(t,x)+\int_{0}^\infty k(s) \Delta v(t-s,x) \dd s 
       \\ & \qquad \qquad  +
      f(t,v(t,x)) + 
      \sqrt{Q} \partial_t W(t,x) \qquad \qquad t >0  \\
    & v(-t,x)=v_0(-t,x), \quad t\geq 0 \\
     & v(t,x)=0, \qquad t\geq 0, x\in \partial \cO.
\end{aligned}
\end{equation}
Our first purpose is to rewrite equation \eqref{eq:uncont} as an evolution equation
defined on a suitable Hilbert space. 

To this end we denote with $L^2(\cO)$ the space of square integrable, real valued functions defined
on $\cO$ with scalar product $\langle u, v\rangle = \int_{\cO} u(\xi)v(\xi) d\xi$, for any $u,v\in L^2(\cO)$.
Sobolev spaces $H^1(\cO)$ and $H^2(\cO)$ are the spaces of functions whose first (resp.
first and second) derivative are in $L^2(\cO)$. We set moreover $H^1_0( \cO)$ the subspace of
$H^1(\cO)$ of functions which vanish (a.e.) on the boundary $\partial \cO$.

We let $X = H^{-1}(\cO)$ the topological dual of $H^1_0( \cO)$.

We recall that the operator $\Delta$ with domain $H^1_0( \cO)$  is the generator of a $C_0$-semigroup of contractions; since $\Delta$ is self
adjoint, the semigroup is analytic: see for instance \cite[22, Theorem 1.5.7, Corollary 1.5.8]{vra03}.

In order to control the unbounded delay interval, we shall consider $L^2$ weighted
spaces. Let 
\begin{align}\label{eq:rho}
 \rho(t) =\int_t^{\infty} k_1(s) ds.
\end{align} Then we set $\cX = L^2_{\rho}(\mathbb{R}_+;H^1_0 (\cO))$ be the space of
functions $y : \mathbb{R}_{+} \rightarrow D(A) = H^1_0 (\cO)$ endowed with the inner product
\begin{equation*}
\langle y_1,y_2\rangle_{\cX}= \int_0^{\infty}\rho(s)\langle\nabla y_1(s), \nabla y_2(s)\rangle \dd s
\end{equation*}
and $\|y\|_{\cX}$ the corresponding norm. On this space, we introduce the delay operator
$K$ with domain $D(K) = \cX$ by setting
\begin{equation*}
K \eta=\int_0^{\infty} k_1(s) \Delta \eta(s) \dd s
\end{equation*}
Finally, we define the Hilbert space $\cH = L^2(\cO) \times L^2_{\rho}(\mathbb{R}_+;H^1_0 (\cO))$
 endowed with the energy
norm $\|x\|^2_{\cH}=\|v\|^2_{L^2(\cO)}+\| \eta\|^2_{ \cX}$, $x=\binom{v}{\eta}$.

Our aim is to reduce this problem to an abstract Cauchy problem on the product space
$\cH$ in such a way that the first component gives the evolution of the 
system while the second contains all the informations concerning the whole history of the solution. 
The state variable in the Hilbert space $\cH$ will be denoted by $X(t)$. Thus $(X(t))_{t\geq 0}$
is a process in $\cH$ and the initial condition is assumed to belong to $\cH$ and satisfies
suitable properties to be precised. 

We introduce the linear operator $\bA$ 
defined as:
\begin{align*}
\bA  \begin{pmatrix}
       v \\ \eta
       \end{pmatrix}  & = 
       \begin{pmatrix}
       \Delta & K \\
          0 & -\partial_s 
       \end{pmatrix} \begin{pmatrix}
       v \\ \eta
       \end{pmatrix} \\& = 
      \begin{pmatrix}
        k_0  \Delta v + \int_0^{\infty}
k_1(s)  \Delta \eta(s) \dd s\\
        - \partial_s \eta\end{pmatrix}
\end{align*}
with domain $D(\bA)\subset \cH$
\begin{align*}
      D(\bA):= \left\{ \begin{pmatrix} v \\ \eta(\cdot)\end{pmatrix} \in H^1_0(\cO)\times W^{1,2}_\rho(\R_+; \hunozero) \, : \, \eta(0)= v, \ \Delta v + K\eta \in \cX\right\}.
\end{align*}

In order to handle the contribution of temperature
values taken in the past, we introduce the new variable
$$\eta^t(s) = v(t-s), \qquad s \geq 0.$$

Moreover, we introduce the non linear operator $\bF: [0,T]\times \cH \to \cH$
\begin{align*}
      \bF \left( t, \begin{pmatrix}
      v \\
      \eta(\cdot)
       \end{pmatrix} \right):=  \begin{pmatrix}
      f(t,v) \\
      0
       \end{pmatrix} 
\end{align*}
where $f$ is the non linear term in Equation \ref{eq:uncont}. Finally we introduce the linear operator $\bQ$ and the stochastic perturbation $\bW$ on $\cH$ as
\begin{align*}
     \bQ := \begin{pmatrix}
      Q & 0 \\
      0 & 0
       \end{pmatrix}
  \qquad \qquad 
     \bW(t)=\begin{pmatrix}
          W(t) \\
          0
      \end{pmatrix}.
\end{align*}
With the above notation,  problem (\ref{eq:uncont}) can be rewritten in the form
\begin{equation}\label{eq:EDS}
 \begin{cases}
\dd X(t) =( \bA X(t) + \bF(X(t)))\dd t +\sqrt{\cQ}\,\dd \bW( t), \qquad t \in [0, T],\\
 X(0)=X_0 \in \cH,
\end{cases}
\end{equation}
where $X(t)$ stands for the pair $\binom{v(t)}{\eta^t(\cdot)}$ and $X_0:=\binom{\bar{v}}{\bar{\eta}(\cdot)}$ is the initial condition. 

In the following (see Section \ref{sec:prop-gen}) we will see that the dynamics of the system is described in terms
of the transition semigroup $e^{t\bA}$ generated by the linear operator $\bA$. As a consequence we will read the solution of the original Volterra equation in the first
component of $X$.

Before prooceding, let us recall the definition of mild solution for the stochastic Cauchy problem
 \eqref{eq:EDS}.
\begin{defn}
    Given an $\cF_t$-adapted cylindrical Wiener process on a probability
space $(\Omega,\cF,(\cF_t)_{t\geq 0},\PP)$, a process $(X(t))_{t\geq 0}$ is a mild solution
of \eqref{eq:EDS} if it belongs to $L^2(0,T;L^2(\Omega;\cH))$ and satisfies $\PP$-a.s. the following integral equation
\begin{align}\label{eq:mild-sol}
    X(t)=e^{t\bA}X_0+ \int_0^t e^{(t-s)\bA}\bF(X(s))\dd s + \int_0^t e^{(t-s)\bA}\sqrt{\bQ}\dd W(s), \qquad t\geq 0.
\end{align}
\end{defn}

Condition \ref{eq:mild-sol} implies that the integrals on the right-member are well defined. In particular, the second integral, which we shall refer to as {\it stochastic convolution}, is a mean- square continuous gaussian process with values in $\cH$. For the analysis of the stochastic 
convolution and its properties, we refer Section \ref{sec:stoc-conv}. 

\subsection{Generation properties}\label{sec:prop-gen}
In this section we are dealing with the generation properties of the leading (matrix) operator
and prove that, in our setting, the operator is quasi-$m$-dissipative (see inequality
\eqref{eq:quasi-diss} below) and that the range of $\mu-\bA$ is dense in
$\cH$ for some (and all) $\mu>0$. In this way we will able to apply the Lumer-Phillips theorem
to conclude that $\bA-\mu$, and hence $\bA$, generates a $C_0$-semigroup.

We start by proving the dissipativity properties.
\begin{thm}
    The operator $(\bA,D(\bA))$ is quasi-$m$-dissipative: for any $\phi=(u,\eta) \in D(\bA)$ there exists $\lambda_0 >0$ such that, for any $\lambda \geq \lambda_0$
  \begin{align}\label{eq:quasi-diss}
      \langle   \bA\phi,\phi \rangle \leq  \lambda \|\phi\|_{\cH}.
\end{align}
\end{thm}

\begin{proof}
    We proceed in the same spirit as Bonaccorsi and Da Prato and Tubaro \cite[Theorem 3.1]{BoDaTu}, but we include the proof for completeness. The difference, here, is that
   we have no conditions linking the constant $k_0$ and the function $k_1$.
   In contrast with \cite{BoDaTu}, this point doesn't allow to prove the pure dissipativity of $\bA$, but only quasi-$m$-dissipativity.

   We compute the scalar product
   \begin{multline*}
         \langle \bA\phi,\phi\rangle =
          k_0 \la \Delta^D u,u\ra_{L^2(\cO)} \\
       +\int_{\R^+} k_1(r) \la u,\Delta^D \eta(r) \ra_{\ldueo}{\rm d}
       r - \int_{\R^+} \rho(r) \la \nabla \eta(r),\nabla \eta^\prime (r)\ra_{\ldueo} {\rm d}r
    \end{multline*}
    and we get
    \begin{align*}
             \langle \bA\phi,\phi\rangle &= -k_0 \|x\|_{\hunozero}^2 -
           \int_{\R_+} k_1(r) \la x,\eta(r)\ra_{\hunozero} \dd r - \int_{\R^+} \rho(r) \frac{1}{2} \frac{\dd}{\dd r}
           \|\eta(r)\|^2_{\hunozero} \dd r \\
         & \leq -k_0 \|x\|_{\hunozero}^2 +  \int_{\R_+} k_1(r) \| x\|_{\hunozero}\|\eta(r)\|_{\hunozero} \dd r - \left.  \frac{1}{2}\rho(r)
           \|\eta(r)\|^2_{\hunozero}\right|_0^{+\infty} \\
       & \quad + \frac{1}{2} \int_{\R_+} \rho^\prime(r)  \|\eta(r)\|^2_{\hunozero} {\rm d} r.
    \end{align*}
    Now recall that $\rho \geq 0$ while $\rho^\prime = -k_1 \leq 0$; choose some $\ep >0$ and
     use the bound $ab -\frac{1}{2}(1-\ep) b^2\leq \frac{1}{2(1-\ep)}a^2$ with $a= \|x\|_{\hunozero}$ and $b=\|\eta(r)\|_{\hunozero}$ to get
   \begin{align*}
          \langle \bA\phi,\phi\rangle & \leq\left(-k_0+ \frac{1-\ep/2}{1-\ep}\rho(0)\right)\|x\|_{\hunozero}^2   + \frac{\ep}{2}\int_{\R^+} \rho^\prime(r) \|\eta(r)\|^2_{\hunozero} \dd r \\
      &\leq \lambda \|\phi\|_\cH^2,
    \end{align*}
    for any $\lambda > \lambda_0:= \min\left\{0,-k_0+ \frac{1-\ep/2}{1-\ep}\rho(0)\right\}$.
\end{proof}

Next, we consider the properties of the resolvent $R(\mu,\bA)$.
\begin{thm}
    For every $\mu>0$ the equation
   \begin{align}\label{eq:risolvente}
       (\mu-\bA)\phi=\psi, \quad \psi \in \cH,
   \end{align}
   has a unique solution $\phi \in D(\bA)$.
\end{thm}

\begin{proof}
    We give only a sketch of the proof, since it essentially repeats the arguments of  Bonaccorsi \& Da Prato \& Tubaro \cite[Proposition 3.2]{BoDaTu}.
  Let $\phi=\binom{u}{\eta}$ and $\psi=\binom{v}{\xi}$. Then equation
    \eqref{eq:risolvente} is equivalent to
  \begin{align*}
       &\mu u- k_0\Delta^D u- K \eta = v\\
      &\mu \eta(s)-T\eta(s)=\xi(s).
  \end{align*}
   From the variation of constant formula we get
   \begin{align*}
        \eta(s)=e^{-\mu s}\eta(0)+\int_0^s e^{-\mu(s-r)} \xi(r)\dd r;
   \end{align*}
   moreover, using the monotonicity property of $\rho$ it is possible to prove that $\eta \in W^{1,2}_\rho(\R_+;\hunozero)$.
    Straightforward calculation, gives
    \begin{align*}
         u= \frac{1}{c_{k,\mu}} R\left( \frac{\mu}{c_{k,\mu}},\Delta^D\right) \tilde{v},
    \end{align*}
      where $c_{k,\mu}$ is a positive constant depending only on $k$ and $\mu$
while $\tilde{v}$ is the function
    \begin{align*}
           \tilde{v}:= v+ \int_0^{+\infty} k_1(r) \Delta^D \int_0^s e^{-\mu(s-r)} \eta(r) \dd s\ \dd r.
   \end{align*}
     Obviously $u\in D(\Delta^D)$.
     Finally, we notice that
    $$
        k_0 \Delta^D u +K \eta = \mu u -v \in L^2(\cO);
 $$
     hence, it turns out that $\phi=\binom{u}{\eta} \in D(\bA)$.
\end{proof}

Taking into account the above results, we can deduce the generation properties for the 
operator $\bA$. Precisely, we have
\begin{prop}\label{prop:gen}
     Under Hypothesis \ref{hp:kernel}, \ref{hp:f} and \ref{hp:noise} the operator 
   $(\bA,D(\bA))$ generates a strongly continuous semigroup.
\end{prop}

\begin{proof}
    The result follows by a direct application of a perturbation method and the Lumer-Phillips Theorem. 
\end{proof}

\section{The stochastic convolution}\label{sec:stoc-conv}

The main object of investigation of this section is the stochastic convolution
corresponding with our problem, that is the process
\begin{align*}
      W_\bA(t):= \int_0^t e^{(t-s)\bA} \sqrt{\bQ}\dd W(s), \qquad t\geq 0.
\end{align*} 

In particular, our purpose is to prove that $(W_A(t))_{t\geq 0}$ is a well-defined mean-square
continuous gaussian process with values in $\cH$.
Following the approach of  Da Prato and Clement \cite{cle-dap97}, Bonaccorsi and Da Prato and Tubaro \cite{BoDaTu}, 
we can give a meaning to the stochastic convolution through the study of the 
 so-called \emph{resolvent family} associated with an abstract homogeneous linear Volterra equation 
of type
\begin{align}\label{eq:eq-risolv}
    v(t) =  k * \Delta^D v(t)
\end{align}
where $k$ is a kernel satsfying Hypothesis \ref{hp:kernel} and where $\Delta^D$ denotes the Laplace operator on $\bar{\cO}$ with Dirichlet boundary conditions.
The concept of the resolvent plays a central role for the theory of linear Volterra equations and can be applied to inhomogeneous problem to derive a variation of parameters formula.  The main tools for the resolvent are described in detail in the monograph \cite{pruss}. In the next
subsection we recall a few basic concepts and results.

\subsection{The resolvent family}\label{ssec:res-fam}

Following \cite[Section 1]{pruss}, we define the resolvent family
 for the equation \eqref{eq:eq-risolv} as
\begin{defn}
    A family $(S(t))_{t\geq 0}$ of bounded linear operators in $X$ is called a resolvent 
   for equation \eqref{eq:eq-risolv}
    if the following conditions are satisfied:
    \begin{enumerate}
         \item[(S1)] $S(0)=I$ and, for all $x\in X$, $t\mapsto S(t)x$ is continuous on $\R^+$;
         \item[(S2)] $S(t)$ commutes with $\Delta^D$, that is for a.e. $t\geq 0$, 
         $S(t) D(\Delta^D) \subset D(\Delta^D)$ and
         \begin{align*}
             \Delta^D S(t) \bar{v} = S(t) \Delta^D \bar{v}, \quad v\in D(\Delta^D);
         \end{align*}
        \item[(S3)] for any $\bar{v}\in D(\Delta^D)$, $t \mapsto S(t)\bar{v}$ is a strong solution 
      of \eqref{eq:uncont} on $[0,T]$, for any $T>0$. 
    \end{enumerate}
\end{defn}
It turns out that if the kernel $k$ satisfies Hypothesis \ref{hp:kernel} (or, more generally,
if it is $\theta$-sectorial for $\theta <\pi$), then equation  \eqref{eq:eq-risolv} admits a resolvent $(S(t))_{t\geq 0}$ which is uniformly bounded in $L^2(\cO)$ (see 
\cite[Corollary 3.3]{pruss}). Consequently (see \cite[Proposition 1.1]{pruss}), problem \eqref{eq:eq-risolv} is well-posed and
its strong solution is given by the function
$v(t)=S(t)\bar{v}$. Besides, since $k(t)$ belongs to $BV_{loc}(\R_+)$, 
$S(t)$ turns out to be differentiable and consequently (by differentiation of equation \eqref{eq:eq-risolv}) 
the function $v$ is the mild solution of 
the homogeneous Cauchy problem
\begin{align}\label{eq:cauchyhom}
    \begin{cases}
    \  v^\prime(t)= \dd k* \Delta^D v(t) \\
   \  v(0)=\bar{v} \in \hunozero.
     \end{cases}
\end{align} 
Here the term $\dd k * \Delta^D v(t)$ denotes the function
\begin{align*}
     \int_0^t k_0 \Delta^D v(s)\delta_0(s) + \int_0^t k_1(t-s) \Delta^D v(s)  \dd s.
\end{align*}
Analogously, it can be proved that if
$g$ is a function belonging to $L^{1}(0,T;X)$,
 then the Cauchy problem
\begin{align}\label{eq:cauchyinh}
    \begin{cases}
    \  v^\prime= \dd k* \Delta^D v + g\\
   \  v(0)=\bar{v}\in \hunozero
     \end{cases}
\end{align} 
is well-posed too and its (unique) mild solution 
can be represented through the variation of parameter formula as 
\begin{align*}
     v(t)=S(t) \bar{v}+ \int_0^t S(t-\tau)g(\tau) \dd \tau, \quad t\geq 0.
\end{align*}
For a full discussion about the notion of well-posedness for equation
\eqref{eq:eq-risolv}, of mild solution for  
problems of type \eqref{eq:cauchyhom}, \eqref{eq:cauchyinh} and their relationship between
the resolvent family we refer to \cite[Section 1]{pruss}. 

Here we want to emphasize that the above arguments can be applied to 
the inhomogeneous Volterra equation {\bf \eqref{eq:uncont}} to obtain existence and
uniqueness of a mild solution and its representation in terms of the resolvent 
family corresponding with the kernel $k=1 * \dd k$. In fact, 
equation {\bf \eqref{eq:uncont}} is equivalent to 
\begin{equation}\label{eq:uncont2}
\begin{aligned}
  &  v_t(t,x) = k_0  \Delta v(t,x) +\int_0^t k_1(s) \Delta v(s,x) \dd s\\
    & \qquad \qquad \qquad + \int_t^\infty k(s) \Delta v(t-s,x) \dd s; \qquad t>0,
\end{aligned}
\end{equation}
with boundary and initial conditions given by:
\begin{align*}
    & v(-t,x)=v_0(-t,x), \quad t\geq0,\\
    &v(t,x)=0, \quad t\geq 0, x\in \partial \cO.
\end{align*}
In abstract form, we have
\begin{align*}
     v^\prime(t) &= \dd k * \Delta^D v(t)+ \int_t^\infty k(s) \Delta^D v_0(t-s) \dd s, \quad t>0\\
     v(0)&=v_0(0) \in \hunozero.
\end{align*}
Therefore, integrating \eqref{eq:uncont2} over $[0,t]$ we obtain
\begin{align}\label{eq:conv1dkD}
    v(t) = v_0(0) + \int_0^t (\dd k * \Delta v)(s)  + 
    \int_0^t \dd s \left(\int_0^\infty k(s+r) \Delta v_0(-r) \dd r\right). 
\end{align}
Now, by the associativity property of the convolution product, the second term in the right member of \eqref{eq:conv1dkD} gives
\begin{align*}
     \int_0^t \dd k * \Delta v(s,x)  &= 1 *( \dd k * \Delta v) 
     = (1* \dd k) * \Delta v\\
    & = k * \Delta v = \int_0^t k(s) \Delta v(t-s) \dd s.
\end{align*}
Hence equation \eqref{eq:uncont} can be rewritten as follows:
\begin{align}\label{eq:Voltlinhom}
     v(t)= v_0(0)+ \int_0^t k(t-s) \Delta v(s) \dd s + h(t)
\end{align}
where the function $h$ is given by
\begin{align*}
    h(t)=\int_0^t \dd s\left( \int_0^\infty k(s+r) \Delta v_0(r) \dd r \right).
\end{align*}
Now the variation of parameters formula implies that the function
\begin{align*}
    v(t)= S(t)v_0(0) + \int_0^t S(t-s)h(s) \dd s
\end{align*}
is a mild solution of the Volterra equation \eqref{eq:uncont2}, provided that 
$h\in L^1(0,T;X)$. We notice that the condition $v_0 \in L^2_\rho(\R^+;\hunozero)$ 
assures the requested regularity for the function $h$.

\subsection{The scalar resolvent family}\label{ssec:screfa}
Suppose that $(S(t))_{t\geq 0}$ is the resolvent family for equation \eqref{eq:eq-risolv}
and
let $\left\{ \mu_j\right\}_{j\in \bN}$ be the set of eigenvalues of $\Delta^D$ with respect
to the basis $\{e_j\}_{j\in\ bN}$. For any $j\in \bN$, we introduce the following one-dimensional Volterra equation
\begin{align}\label{eq:risolvente2}
    s_j(t)+ \mu_j (k * s_j)(t)  =1.
\end{align}

Then (see \cite[Section 1.3]{pruss}) a unique solution to \eqref{eq:risolvente2} exists and it satisfies
\begin{align*}
    S(t)e_j=s_j(t)e_j, \qquad t\geq0.
\end{align*}
In particular, the resolvent family $S(t)$ admits a decomposition in the basis $\{e_j\}$ of $L^2(\cO)$ in terms of the solutions $s_j$ to \eqref{eq:risolvente2}. 

In the sequel we state and prove some useful estimates on the scalar resolvent functions $s_j$. They are crucial to study the stochastic convolution and descend immediately from the assumption on the kernel $k$.

\begin{lem}\label{lem:sol-sj}
   Let $k$ satisfy Hypothesis \ref{hp:kernel}. Then, for any $j\in \bN$, equation \eqref{eq:risolvente2} admits a solution $s_j(t)$ such that
the following properties hold: 
\begin{enumerate}
\item \label{it:sj}  $ | s_j(t) |\leq M$ for all $t>0$;
\item \label{it:sj'}  
      $\int_0^\infty |s^\prime_j (t)| \dd t \leq C$;
\item \label{it:tsj'} 
      $\int_0^\infty t|s^\prime_j (t)|\dd t \leq C \mu_j^{-1}$;
\item \label{it:sjprimo} 
   $\int_0^{+\infty} |s_j(t)|{\rm d}t \leq  C\mu_j^{-1}$.
where $M$ and $C$ denotes suitable positive constants.
  \end{enumerate}
\end{lem}

\begin{proof}
    Assertion \ref{it:sj} follows from \cite[Corollary 3.3]{pruss}, while assertions \ref{it:sj'}
   and \ref{it:tsj'} are contained in Monnieaux and Pruss \cite[Proposition 6]{mon-pru97}
   (observe the relation $s^\prime_j(t)=-\mu_j r_j(t)$ to connect the notations).
   To prove \ref{it:sjprimo}, we notice that
   $$
        s_j(t)= s_j(R)-\int_t^R s_j^\prime(\tau) \dd \tau.
   $$
    Hence assertion \ref{it:sj'} implies that the limit of $s_j(R)$ for $R\to \infty$ exists; moreover,
    we have
   \begin{align*}
       \lim_{R \to \infty}s_j(R)= s_j(0)+ \lim_{R \to \infty}\int_0^R s_j^\prime(\tau) \dd \tau = 1+\int_0^\infty s_j^\prime(\tau) \dd \tau.
   \end{align*}
   We observe that the last term in the above equality can be rewritten as
    \begin{align*}
          \lim_{\lambda \to 0^+} \int_0^\infty e^{-\lambda \tau}s_j^\prime(\tau) \dd \tau 
          = 
         & \lim_{\lambda \to 0^+}\left. s_j(\tau) e^{-\lambda \tau} \right|_{0}^\infty
          + \lim_{\lambda \to 0^+} \lambda \hat{s}_j(\lambda)  \\
          = \ & -1 + \lim_{\lambda \to 0^+} \lambda \hat{s}_j(\lambda).
   \end{align*}
Further, since $s_j$ satisfies equation \eqref{eq:risolvente2}, we get
\begin{align}\label{eq:trLsj}
    \lambda \hat{s}_j(\lambda)= \frac{1}{\lambda + \mu_j \hat{k}(\lambda)} =
   \frac{\lambda}{\lambda^2 + k_0+ \hat{k}_1 (\lambda)};
\end{align}
in fact, we have
$$
     \hat{s}_j (\lambda) + \mu_j \hat{k}(\lambda)\hat{s}_j(\lambda) =\frac{1}{\lambda},
$$
and  
$$
     \hat{k}(\lambda)= \frac{k_0}{\lambda} + \frac{\hat{k}_1(\lambda)}{\lambda}.
$$
We notice that, since $k_1$ belongs to $L^1(\R^+)$, for any $\lambda \geq 0$ it holds
\begin{align*}
      |\hat{k}_1(\lambda)|=\left|\int_0^\infty e^{-\lambda t} k_1(t) \dd t\right|
       \leq \int_0^\infty k_1(t) \dd t <\infty.
\end{align*}
Taking into account the last inequality and equality \eqref{eq:trLsj} we see that
the limit  of $s_j(R)$ for $R\to \infty$ satisfies
\begin{align*}
      \lim_{R\to \infty} s_j(R)= \lim_{\lambda \to 0^+} \frac{\lambda}{\lambda^2 + k_0+ \hat{k}_1 (\lambda)} =0.
\end{align*}
Therefore
\begin{align*}
     s_j(t)=-\int_t^\infty s^\prime_j(\tau)\dd \tau
\end{align*}
yields
\begin{align*}
     \int_0^\infty |s_j(\tau)| \dd \tau \leq \int_0^\infty \int_t^\infty |s^\prime_j(\tau)|\dd \tau \, \dd t = \int_0^\infty \tau |s_j^\prime(\tau)| \dd \tau \leq C\mu^{-1/\delta},
\end{align*}
by assertion \ref{it:tsj'}.
\end{proof}

For further use, we conclude this subsection with an estimate concerning the norm
of $s_j$ in $L^2(\R^+)$.
\begin{lem}\label{lem:sj2}
    Suppose that the kernel $k$ is subject to Hypothesis \ref{hp:kernel}. Then for each
    $\theta \in (0,1)$ and for any $T>0$,, there exists a constant $C_{\theta,T} >0$ (depending only on $\theta$ and $T$) such that 
    \begin{align*}
        \int_0^T \dd \tau \int_0^\tau s_j^2(\sigma) \dd \sigma \leq C_{\theta,T} \mu_j^{-(\theta+1)/\delta}. 
   \end{align*}
\end{lem}
\begin{proof}
    From assertion \ref{it:sj} of Lemma \ref{lem:sol-sj} we obtain
    \begin{align*}
          \int_0^\tau s_j^2(\sigma) \,\dd \sigma \leq M \int_0^\tau |s_j(\sigma)|\dd \sigma
           <M \mu_j^{-1/\delta},
\end{align*}
  as well as
   \begin{align*}
        \int_0^\tau s_j^2(\sigma) \, \dd \sigma \leq M^2 \tau;
\end{align*}
hence employing Lemma \ref{lem:sol-sj} - \ref{it:sjprimo}
\begin{align*}
      \int_0^\tau s_j^2(\sigma) \dd \sigma \leq M^{-2\theta} \tau^{-\theta} \left(\int_0^\tau |s_j(\sigma)| \dd \sigma \right)^{1+\theta}  \leq C_\theta\mu_j^{-(1+\theta)/\delta} \tau^{-\theta}.
\end{align*}
Now integrating both members of the previous inequality we obtain the thesis.
\end{proof}

\subsection{The representation of the semigroup}
In the following we show that the semigroup  
corresponding with the linear operator $e^{t\bA}$ can be computed explicitly in terms of the \emph{resolvent family}
$(S(t))_{ t \geq 0}$. 

We recall that since the linear operator 
$\bA$ generates a $C_0$-semigroup, there exists a unique mild solution $(X(t))_{t\geq 0}$
for the deterministic equation 
\begin{equation}\label{eq:det}
  \begin{aligned}
   \begin{cases}
     \ X^\prime (t)= \bA X(t),\\
    \ \begin{pmatrix}
      v(0)\\
       \eta^0(\cdot)
    \end{pmatrix}   = 
     \begin{pmatrix}
      \bar{v}\\
       \bar{\eta}(\cdot)
    \end{pmatrix}  \in D(\bA).
\end{cases}
  \end{aligned}
\end{equation}
The variation of parameter formula for abstract evolution equations applies to equation
\eqref{eq:det} and we can write:
\begin{align*}
     X(t)= e^{t\bA}X_0.
\end{align*}
If we set 
\begin{align*}
     e^{t\bA}:= \begin{pmatrix}
        e^{t \bA }_{11} & e^{t \bA }_{12}  \\
               e^{t \bA }_{21}  & e^{t \bA }_{22}  
    \end{pmatrix},
\end{align*}
then, for any $t>0$, we have
\begin{equation}\label{eq:sol-sgr}
   \begin{aligned}
    v(t) &= e^{t \bA }_{11} (t)\bar{v} + e^{t \bA }_{12}(t)\bar{\eta};\\
  \eta^t(\cdot)& =e^{t \bA }_{21} (t)\bar{v}+ e^{t \bA }_{22}(t)\bar{\eta}.
\end{aligned}
\end{equation}

By construction, the first component of $X$ satisfies the inhomegenous Volterra equation
\begin{align} \label{eq:Voltlinhom2}
     v(t)= v_0(0)+ \int_0^t k(t-s) \Delta v(s) \dd s +
    \int_0^t \dd s\left( \int_0^\infty k(s+r) \Delta v_0(-r) \dd r \right)
\end{align}
and the variation of parameters formula for Volterra equations applied to \eqref{eq:Voltlinhom2}  (see Subsection \ref{ssec:res-fam}) yields
 \begin{equation}\label{eq:generazione-sem}
\begin{aligned}
v(t) &= S(t)\bar{v} + \int_{0}^t S(t-\tau )h(\tau ) \dd\tau \\
\eta^t(s) &= \left\{
          \begin{array}{ll}
            v(t- s) = S(t- s)\bar{v} + \int_{0}^{t-s} S(t-s-\tau )h(\tau ) \dd\tau, & 0 < s \leq t \\
            \bar{\eta}(t - s) , & s > t,
          \end{array}
        \right.
\end{aligned}
\end{equation}
where
\begin{align*}
h(t)&=\int_{\R_+} k(t + \sigma)\Delta v_0 (-\sigma) \dd \sigma  \\
    &=\int_{\mathbb{R}_+}k(t + \sigma)\Delta \bar{\eta}(\sigma) \dd\sigma.
\end{align*}

%
%

Comparing the first terms in equalities \eqref{eq:sol-sgr} and \eqref{eq:generazione-sem}, we obtain
\begin{equation}\label{eq:sgr1}
e^{t \bA }_{11}\bar{v} = S(t)\bar{v} \qquad 
e^{t \bA }_{12} \bar{\eta}=  \int_{0}^t S(t-\tau )h(\tau ) d\tau.
\end{equation}
Moreover, from the second part of \eqref{eq:sol-sgr} and \eqref{eq:generazione-sem}, we have for $s\geq 0$
\begin{equation}\label{eq:sgr2}
\begin{aligned}
&(e^{t \bA }_{21}\bar{v})(s)=S(t-s)\bar{v}\,\mathbf{1}_{[0,t]}(s)
\\
&(e^{t \bA }_{22}\bar{\eta})(s)=\left\{
                           \begin{array}{ll}
                             \int_{0}^{t-s} S(t-s-\tau)f_y(\tau ) \dd \tau, & 0< s\leq t  \\
                             \bar{\eta}(t -s), & s>t.
                           \end{array}
                       \right.
\end{aligned}
\end{equation}
Thus the semigroup $e^{t\bA}$ is completely described in terms of the 
resolvent family. As we will  see in the next subsection, the above characterization
allows to study the stochastic convolution process. 

\subsection{The stochastic convolution}
We are now in the position to prove the main result of this section.
We recall that $(W(t))_{t\geq 0}$ is a cylindrical Wiener process 
of the form 
\begin{align*}
       \la W(t),x\ra = \sum_{k=0}^\infty \la e_k,x \ra \beta_k(t), \qquad t\geq 0, x\in L^2(\cO)
    \end{align*}
     where $\left\{ \beta_k\right\}_{k\in\bN}$ is a sequence of real, standard, independent
     Brownian motions on $(\Omega, \cF, (\cF_t)_{t\geq 0}, \PP)$. 
We have:
\begin{lem}\label{lem:stoc-conv}
      Under Hypothesis \ref{hp:noise}, for all $T>0$ the process $(W_\bA(t))_{0\leq t\leq T}$
defined as
\begin{align}\label{eq:stoc-conv}
     \bW_\bA(t):=\int_0^t e^{(t-s)\bA} \dd \bW(s),
\end{align}
is a gaussian random variable with mean $0$ and covariance operator $$
    \cQ_t:= \int_0^t e^{sA} \cQ e^{s\bA^*} \dd s.$$
\end{lem}
\begin{proof}
     It is well-known that the thesis follows provided that
   \begin{align*}
         \int_0^t \|e^{\tau\bA}\cQ \|^2_{HS} \dd \tau <C_T,
\end{align*}
      where $C_T$ is a positive constant depending only on $T>0$.
    Recalling the representation of $e^{t\bA}$ given in \eqref{eq:sgr1} and \eqref{eq:sgr2}, we have
   that
    \begin{equation}\label{eq:series}
    \begin{aligned}
       &\int_0^t \|e^{\tau\bA}\cQ \|^2_{HS} \dd \tau = \sum_{j=1}^\infty
       \int_0^t \left|e^{\tau\bA} \begin{pmatrix} \sqrt{Q} \\ 0\end{pmatrix} \right|^2_{\cH}\dd \tau \\
     & \quad = \sum_{j=1}^\infty \int_0^t  \left|\begin{pmatrix} S(\tau)\sqrt{\lambda_j} e_j  \\ S(\tau-\cdot ) \sqrt{\lambda_j} e_j {\bf 1}_{[0,\tau]}(\cdot)\end{pmatrix} \right|^2_{\cH} \\
      &\quad = \sum_{j=1}^\infty \lambda_j \int_0^t \|S(\tau)e_j\|^2_{L^2(\cO)} \dd \tau+ \\
        &\qquad \quad \sum_{j=1}^\infty \lambda_j \int_0^t  \int_{0}^\infty \delta(\sigma)
         \|S(\tau-\sigma)e_j\|^2_{\hunozero}{\bf 1}_{[0,\tau]} \dd \sigma\, \dd\tau.
    \end{aligned}
   \end{equation}
   We consider separately the two series in the previous formula. We recall that
     $S(t)e_j= s_j(t) $ for any $j\in \bN$ (see Subsection \ref{ssec:screfa}); hence we get
    \begin{align*}
        \sum_{j=1}^\infty \lambda_j \int_0^t \|S(\tau)e_j\|^2_{L^2(\cO)} \dd \tau =
        \sum_{j=1}^\infty \lambda_j \int_0^t |s_j(\tau)|^2 \dd \tau
       \leq \sum_{j=1}^\infty \lambda_j \int_0^t |s_j(\tau)| \dd \tau,
     \end{align*}
    where the last inequality follows from Lemma \ref{lem:sol-sj}, point \ref{it:sj}.
    Moreover, since it holds also that $\int_0^\infty |s_j(\tau)| \dd \tau < (\mu_jk_0)^{-1}$
    (see Lemma \ref{lem:sol-sj}, point \ref{it:sjprimo}), it follows that
   \begin{align*}
        \sum_{j=1}^\infty \lambda_j \int_0^t \|S(\tau)e_j\|^2_{L^2(\cO)} \dd \tau
       \leq \frac{1}{k_0}\sum_{j=1}^\infty\frac{ \lambda_j }{\mu_j}.
    \end{align*}
     Concerning the second series in \eqref{eq:series}, applying Fubini's theorem, we get
    \begin{align*}
        \sum_{j=1}^\infty \lambda_j \mu_j \int_0^t \int_0^\tau \delta(\sigma) |s_j(\tau-\sigma)|^2 \dd \sigma \, \dd \tau
      \leq  \sum_{j=1}^\infty \lambda_j \mu_j \int_0^t \dd \sigma \delta (\sigma) 
      \int_\sigma^t \dd \tau |s_j(\tau-\sigma)|^2 
\end{align*}
and, taking into account Lemma \ref{lem:sj2} and the definition of the function $\rho$
  (see \eqref{eq:rho}),
\begin{align*}
      \sum_{j=1}^\infty \lambda_j \mu_j \int_0^t \int_0^\tau \rho(\sigma) |s_j(\tau-\sigma)|^2 \dd \sigma \, \dd \tau 
   & \leq \sum_{j=1}^\infty \lambda_j \mu_j  \int_0^t \rho(0) C_\theta\mu_j^{-(1+\theta)/\delta}(\tau-\sigma)^{-\theta}  \\
   & \leq C_\theta T^{1-\theta}\rho(0) \sum_{j=1}^\infty \frac{\lambda_j}{\mu_j^{(1+\theta)/\delta-1}}.
    \end{align*}
    By the above estimates and condition \ref{it:AQ} in Hypothesis \ref{hp:noise},
    we conclude that, for any $\theta \in (0,1)$ such that $1+\theta>\delta$,
    \begin{align*}
          \int_0^t \|e^{\tau\bA}\cQ \|^2_{HS} \dd \tau \leq C_T,
    \end{align*}
    where $C_T:=C_\theta T^{1-\theta} \rho(0) {\rm Tr}[Q(-\Delta)^{(1+\theta)/\delta -1}]$.
\end{proof}

\section{Existence and uniqueness}\label{sec:exun}
In this section we aim to prove existence and uniqueness of the solution for the 
uncontrolled equation
\begin{equation}\label{eq:Voltexun}
\begin{aligned}
     &\partial_t v(t,x)= k_0 \Delta v(t,x)+\int_{0}^\infty k_1(s) \Delta v(t-s,x) \dd s \\
    &\quad \qquad \qquad \qquad + f(v(t,x)) + 
     \sqrt{Q} \partial_t W(t,x), \qquad \qquad t> 0\\
    &v(s,x)=v_0(s,x), \quad \quad s\leq 0,\\
    &v(t,x)=0, \quad \quad t\geq 0,\ x\in \partial \cO.
\end{aligned}
\end{equation} 
where the coefficients $k_0,k_1,f,Q$ satisfy the assumptions made in Section \ref{sec:genass}.

Recalling what has been showed in the previuos section, the above equation as an abstract equation on the space $\cH:=L^2(\cO) \times L^2_\rho(\R_+;\hunozero)$ 
\begin{equation}\label{eq:abst}
\begin{cases}
\dd X(t) = \bA X(t)\dd t + \bF(X(t))\dd t + \sqrt{\cQ} \, \dd W(t)\quad \quad t >0,\\
  X(0)=X_0.
\end{cases}
\end{equation}
We recall that, from Proposition \ref{prop:gen} $\bA$ is the generator of a $C_0$-semigroup, while from the assumption on the function $f$ we get that $\bF: \cH \to \cH$ is Lipschitz continuous. Moreover, $\bQ$ is a linear operator on $\cH$ involving the covariance operator $Q$, $X_0= (v_0(0,\cdot),  (v_0(-s,x))_{s\geq 0})^t$ and 
the stochastic convolution $W_\bA(t)$ introduced in 
\eqref{eq:stoc-conv} is a well-defined gaussian process (see Lemma \ref{lem:stoc-conv}). 


Existence and uniqueness of mild solution for the abstract evolution equation \eqref{eq:abst}
is a classical result within the theory of stochastic equation in infinite dimension.
The proof follows from a fixed point argument and can be found in \cite[ Theorem 7.2]{dpz:Stochastic}
\begin{thm} \label{thm:exun} 
 For arbitrary $T > 0$, and any $X_0 \in \cH$ there exists a unique mild solution $(X(t))_{t\geq 0}$ of equation \eqref{eq:abst} which belongs to the space  
$L^p(\Omega; C([0; T];\cH))$ for any $p\geq 1$.
\end{thm}

An immediate consequence of the above result is that also the original stochastic Volterra 
equation \eqref{eq:Voltexun} admits a unique mild solution. The definition of mild solution
involves the resolvent family introduced in Section \ref{ssec:res-fam} and reads as follows
\begin{defn}
     A $L^2(\cO)$-valued process $(v(t))_{t\geq 0}$ is a mild solution of the stochastic Volterra equation of \eqref{eq:Voltexun} if $v\in L^2(0,T;L^2(\Omega;L^2(\cO)))$ and satisfies
\begin{multline*}
     v(t)=S(t)v_0(0)+\int_0^t S(t-s)  f(s,v(s)) \dd s \\
       \int_0^t S(t-s) \int_{\R_+} k(s+r)\Delta v_0(-r)\dd r 
      +\int_0^t S(t-s) \sqrt{Q}\dd W(s) 
\end{multline*}
\end{defn}
\begin{thm}
     For arbitrary $T > 0$ and any $v_0 \in L^2_\rho(\R_+;\hunozero) $ there exists a unique mild solution
$v=v(t), \ t\geq 0$ of equation \eqref{eq:Voltexun}  which belongs to the space 
   $C_\cF([0,T];L^2(\cO))$.
\end{thm}

\begin{proof}
    The proof follows directly from Theorem \ref{thm:exun}. In fact, the mild solution of \eqref{eq:Voltexun} is represented by the first component of the process
    $(X(t))_{t\geq 0}$.

\end{proof}
\section{Synthesis of the optimal control}\label{sec:cont}
In this section we proceed
with the study of the optimal control problem associated with the stochastic Volterra
equation 
\begin{equation}\label{eq:Volt-contr}
  \begin{aligned}
     &\partial_t v(t,x) =k_0 \Delta v(t,x)+\int_{-\infty}^t k_1(t-s) \Delta v(s,x) \dd s \\
    & \qquad \qquad \qquad  +
      f(t,v(t,x)) + \sqrt{Q}(r(t,v(t,x),\gamma(t,x))+\partial_t W(t,x)),
\end{aligned}
\end{equation}
in the bounded domain $\cO\subset \R^d$, with Dirichlet boundary condition $v(t,x)=0, t\in [0,T], x\in \partial \cO$
and initial condition $v(t,x)=v_0(t,x), t\leq 0,x\in \cO$.
Here $f$ is the nonlinear function introduced in Hypothesis \ref{hp:f}
and $\gamma=\gamma(\omega,t,x)$ is the control variable, which is assumed to be a predictable
real-valued process
$\cF_t$-adapted.
The optimal control that we wish to treat consists in minimizing
over all admissible controls
a cost functional of the form
\begin{align*}
      \bJ(v_0,\gamma):= \bE \int_0^T \int_{\bar{\cO}} \ell (t,v(t,\xi), \gamma(t,\xi))\dd \xi \, \dd t
       + \bE \int_{\bar{\cO}} \phi(v(T,\xi)) \dd \xi,
\end{align*}
where $\ell$ and $\phi$ are given real-valued functions. 

We will work under the following general assumptions.
Concerning the function $r, \ell, \phi$ we require:
\begin{hypothesis}\label{hp:rlphi}
    \begin{enumerate}
         \item[]
         \item $r: \ [0,T] \times \R \times \R \to \R$ and $\ell: \ [0,T] \times \R \times \R \to \R$ are measurable functions and there exists $m\in \bN$ such that for a.e. $t\in [0,T]$ and for $\theta_1,\theta_2, y$ in $\R$,
    \begin{align*}
        |r(t,x_1,y)-r(t,x_2,y)| +& |\ell(t,\theta_1,y)-\ell(t,\theta_2,y)| \\
        &\leq C(1+|\theta_1|+|\theta_2|)^m|\theta_1-\theta_2|,\\
       |r(t,\theta_1,y)|+ |\ell(t,0,y)| &\leq C.
    \end{align*}
      \item $\phi \in C^1(\R)$ and there exist $L>0$ and $k\in \bN$ such that for every $\theta\in \R$
       $$
             |\phi^\prime(\theta)| \leq L(1+|\theta|)^k.
      $$
    \end{enumerate}
\end{hypothesis}

In order to characterize the optimal control through a feedback law, we impose the following
additional condition on the nonlinear term $f$:
\begin{hypothesis} \label{hp:f+cont} 
The function  $f:\  [0, T ] \times  \R \to \R$ is measurable, for every $ t \in [0, T ]$ the function $f(t, \cdot):\  \R \to\R$  is
continuously differentiable and there exists a constant $C_f$ such that
$$
     \left| \frac{\partial}{\partial x}f(t,\xi) \right| \leq C_f, \qquad t\in [0,T], \ \xi\in \R.
$$
\end{hypothesis}

To handle the control problem, we first restate equation \eqref{eq:Volt-contr} in an
evolution setting and we provide the synthesis of the optimal control by using the forward-backward system approach.

Arguing as in Section \ref{sec:anal}, given a control process $\gamma$ and any $t\in [0,T], \  v_0\in L_\rho^2(\R_+;\hunozero)$ we rewrite the problem \eqref{eq:Volt-contr}
in the following abstract form
\begin{equation}\label{eq:Volt+cont}
\begin{cases}
    \dd X(t) = \bA X(t) \dd t + \bF (t,X(t))\dd t + \sqrt{\bQ}( R(t,X(t),\gamma(t)) \dd t+ \dd W(t)) \\
    X(0)=X_0,
\end{cases}
\end{equation}
where $X_0=( v_0(0), v_0(\cdot))$ and $R: \ [0,T] \times \cH \times \cX \to \cH$ is the mapping defined by
\begin{align*}
          R\left(t,\binom{v}{\eta},\gamma\right)= \begin{pmatrix}
              r(t,v, \gamma) \\ 0 
     \end{pmatrix}, \quad t\in [0,T], \ \binom{v}{\eta} \in \cH, \gamma \in \cX.
\end{align*} 
%
%
%
In this setting the cost functional will depend on $X_0$ and $\gamma$ and is given by
\begin{align}\label{eq:cost}
  \bJ (X_0, \gamma) = \bE \int_0^T L(t,X(t),\gamma(t)) \dd t + 
\bE [\Phi(X(T ))]
\end{align}
where $L: \ [0,T] \times \cH \times \cX \to \R$ is given by
\begin{align*}
    L\left(t,\binom{v}{\eta},\gamma\right)= \int_{\bar{\cO}} \ell(t,v(\xi),\gamma(\xi)) \dd \xi
\end{align*}
for any $t>0, \ \binom{v}{\eta} \in \cH, \gamma \in \cH$ and
$\Phi: \ \cH \to \R$ is defined as
\begin{align*}
    \Phi\binom{v}{\eta}= \int_{\bar{\cO}}  \phi(v(\xi)) \dd \xi, \quad \binom{v}{\eta}\in \cH.
\end{align*}
There are
different ways to give a precise meaning to the above problem; one of them is the so called 
\emph{weak formulation} and will be specified below.

In the weak formulation the class of \emph{admissible control systems} (a.c.s.) is given by 
the set 
$\bU:=(\hat{\Omega}, \hat{\mathcal{F}},(\hat{\mathcal{F}_t})_{t \geq 0}, \hat{\PP}, \hat{W},\hat{\gamma})$, where $(\hat{\Omega}, \hat{\mathcal{F}}, \hat{\PP})$ is a complete probability space; the filtration $(\hat{\mathcal{F}}_t)_{t \geq 0}$ verifies the usual conditions, the process $ \hat{W}$ is a Wiener process with respect to the filtration
 $(\hat{\mathcal{F}}_t)_t \geq 0)$ and the control  $\hat{\gamma}$ is an $\cF_t$-predictable process taking value in some subset $\cU$ of $\cX$ with respect to the filtration $(\hat{\mathcal{F}}_t)_{t \geq 0})$.

With an abuse of notation, for given $X_0 \in \cH$, we associate to every a.c.s. a cost functional 
$\bJ(x,\bU)$ given by the right side of \eqref{eq:cost}.
Altough formally the same, it is important to note that now the cost is a functional of the a.c.s.
and not a functional of $\hat{\gamma}$ alone.
Any a.c.s. which minimizes $\bJ(x, \cdot)$, if it exists, is called optimal for the control problem
starting from $X_0$ at time $t$ in the weak formulation. The minimal value of the cost is then called
the optimal cost.
Finally we introduce the value function $V : [0, T] \times \cH \to \R$ of the problem as:
\begin{align*}
   V( X_0) =\inf_{\gamma \in \cU} \bJ(X_0, \gamma),\quad   \ X_0\in \cH,
\end{align*}
where the infimum is taken over all a.c.s. $\bU$.

At this moment it is convenient to list the relevant properties of the objects introduced so far in this section.
Therefore we formulate the following proposition.
\begin{prop}\label{prop:cont}
     Under Hypothesis  \ref{hp:kernel},\ref{hp:f}, \ref{hp:f+cont}, \ref{hp:noise} and \ref{hp:rlphi} the following properties hold:
     \begin{enumerate}
        \item The functions $R$ and $L$ are Borel measurable and there exist constants $C$, $m,k\in \bN$ such that for any $t>0, X_1,X_2 \in \cH$ and $\gamma \in \cU$
           \begin{align*}
        |R(t,X_1,\gamma)-R(t,X_2,\gamma)| +& |L(t,X_1,\gamma)-L(t,X_2,\gamma)| \\
        &\leq C(1+|X_1|+|X_2|)^m|X_1-X_2|,\\
       |R(t,X_1,\gamma)|+ |L(t,0,\gamma)| &\leq C.
    \end{align*}
      \item $\Phi$ is G\^ateaux differentiable and there exist $C_\Phi>0$ and $k\in \bN$ such that for every $X_1,X_2\in \cH$
       $$
             |\Phi(X_1)-\Phi(X_2)| \leq C_\Phi|X_1-X_2|
      $$
        \item $\bF: \ [0, T ] \times \cH \to \cH$ is a measurable function and there exists a constant $C_\bF$ such that
\begin{align*}
    |\bF(t, 0)| \leq  C_\bF,\quad  |\bF(t, X_1) - \bF(t, X_2)| \leq  C_{\bF} |X_1 - X_2|,  
\end{align*}
for every $t\in [0, T ],\ X,X_1, X_2 \in \cH$.
Moreover, for every  $t \in [0, T ]$, $\bF(t, \cdot)$ has a G\^ateaux derivative $\nabla \bF(t, X)$ at every point $X \in \cH$. Finally,
the function $(X,H)\mapsto \nabla \bF(t, X)[H]$ is continuous as a map $\cH \times \cH \to \R$.
     \end{enumerate}
\end{prop}

Optimal control problems associated with equation \eqref{eq:Volt+cont} and the cost functional 
\eqref{eq:cost} when the coefficients has the properties listed in Proposition 
\ref{prop:cont} has
been exhaustively studied by Fuhrman and Tessitore in \cite{FuTe/2002}, compare Theorem 7.
Within their approach the existence of an optimal control is related to the existence
of the solution of a suitable forward backward system (FBSDE) that is a system in which the coefficients of the  backward 
equation depend on the solution of the forward equation.
Moreover, the optimal control can be selected using a feedback law given in terms of the
solution to the corresponding FBSDE.

We introduce the hamiltonian 
function $\psi: \ [0,T] \times \cH \times \cH \to \cH$ setting 
\begin{align*}
     \psi(t,X,Z)= \inf_{\gamma\in \cU}\left\{L(t,X,\gamma)+\langle Z, R(t,X,\gamma)\rangle \right\}, \qquad 
    t\in [0,T], \ X\in \cH, \ Z \in \cH,
\end{align*}
and we define the following set
\begin{equation}\label{eq:Gamma}
\begin{aligned}
   \Gamma(t,X,Z)= \left\{\gamma\in \cU: L(t,X,\gamma)+\langle Z, R(t,X,\gamma)\rangle= \psi(t,X,Z)\right\},\\
   \qquad \qquad \quad \qquad t\in [0,T], \ X \in \cH, \ Z \in \cH.
\end{aligned}
\end{equation}

For further use we require some additional properties of the function $\psi$:
\begin{hypothesis}\label{hp:psi}
    \begin{enumerate}
   \item[]
   \item For all $t\in [0,T]$, for all $X,Z\in \cH$ there exists a unique $\Gamma(t,X,Z)$
that realizes the minimum in 
\eqref{eq:Gamma}. Namely:
     \begin{align*}
       \psi(t,X,Z)=L(t, X,\Gamma(t,x,Z))+\langle Z,r(t,X,\Gamma(t,X,Z)\rangle
     \end{align*} 
    with $\Gamma \in C([0,T] \times \cH \times \cH;\cU)$.
     \item For almost
every $ s\in [0, T ]$ the map $\psi(s, \cdot, \cdot)$ is G\^ateaux differentiable on $\cH \times \cH$ and the maps $(X, H, Z) \mapsto \nabla_X \psi(s, X, Z)[H]$
and $(X, Z, K) \mapsto \nabla \psi(s, X, Z)[K]$ are continuous on $\cH \times \cH \times\cH$ and 
$\cH \times \cH \times \cH$ respectively.
  \end{enumerate}
\end{hypothesis}
\begin{rem}
 It is easy to prove that combining the previous assumption with Proposition \ref{prop:cont} we can deduce the following 
properties of $\psi$:
\begin{enumerate}
\item $\psi$ is a measurable mapping and there exists a constant $C$ such that 
\begin{align*}
   |\psi(t, X_1, Z) - \psi(t, X_2, Z)|\leq  C (1 + |X_1| + |X_2|)|X_2 -X_1| 
\end{align*}
for all $X_1,
X_2,Z \in \cH$ and $t \in [0, T ]$.
\item Setting $C_\cU := \sup\left\{ |\gamma| : \ \gamma \in \cU\right\}$
   we have $$|\psi(s, X, Z_1) -\Psi(s, X, Z_2)| \leq  C_\cU |Z_1 - Z_2|,$$ for every $s \in [0, T ]$,
$X, Z_1, Z_2 \in \cH$. 

Finally,  $\sup_{s\in [0,T ]}
|\psi(s, 0, 0)| \leq C$.
\end{enumerate}
\end{rem}

Now,
let us consider an arbitrary set-up $(\tilde{\Omega},\tilde{\cF},\tilde{\PP},\tilde{W})$ and 
\begin{equation}\label{eq:cont2}
\begin{aligned}
   \tX(t)=e^{t\bA} \tX_0 +
\int_0^t e^{(t-\sigma)\bA}\bF(\sigma, \tX_\sigma) \dd \sigma +
\int_0^t e^{(t-\sigma)}\sqrt{\bQ} \dd \btW(\sigma), \quad  t\in  [0, T], 
\end{aligned}
\end{equation}
where $\btW(t)=(\tilde{W}(t), \, 0)^t$. By Theorem \ref{thm:exun} stated in Section \ref{sec:exun}, equation \ref{eq:cont2} is well-posed and the solution $(\tX(t))_{t\geq 0}:$ is a continuous process in $\cH$, adapted to the filtration $(\tF_t)_{t \geq 0}$.
Moreover, the law of $(\btW, \tX)$ is uniquely determined by $X_0$, $\bA$, $\bF$ and $\sqrt{\bQ}$. 
We define the process
\begin{align*}
   \bW^\bU(t) = \tilde{\bW}(t) - \int_0^t R(s,\tX(s), \hat{\gamma}(s)) \dd s, \quad t\in [0, T], 
\end{align*}
and we note that, since $R$ is bounded, by the Girsanov theorem there exists a probability measure $\PP $ on $(\Omega,\cF)$
such that $\bW^\bU$ is a Wiener process under $\PP$. Rewriting equation \eqref{eq:cont2}
in terms of $\bW^\bU$ we get that
$\tX$ solves the controlled state equation (in weak sense)
\begin{equation}
\begin{aligned}
&\tX(t) = \tX_0 + \int_0^t e^{(t-\sigma)\bA}\bF(\sigma, \tX_\sigma) \dd \sigma +\\
&\qquad\int_0^t e^{(t-\sigma)\bA}\sqrt{\bQ} \dd \bW^\bU(\sigma) +
\int_0^t e^{(t-\sigma)\bA}R(s,\tX(s) , \hat{\gamma}(s)) \dd s.
\end{aligned}
\end{equation}
Next we consider the backward
stochastic differential equation
\begin{align}\label{eq:bsde}
  \tY(t) + \int_t^T \tZ \dd \tW(\sigma) = \Phi(\tX(T)) + \int_t^T
\psi(\sigma, \tX(\sigma), \tZ(\sigma))\dd \sigma, \quad  t \in \ [0, T],
\end{align}
where $\psi$ is the hamiltonian function and $\Phi$ is the function defining the final cost.
 Under our assumptions , we can apply \cite[Proposition 3.2 and Theorem 4.8]{FuTe/2002}
and state that there exists a solution $(\tX,\tY, \tZ)$ of the forward-backward system
\eqref{eq:cont2}-\eqref{eq:bsde} on the interval $[0, T]$, where $\tY$ is
unique up to indistinguishability and $\tZ$ is unique up to modification. Moreover from the proof
of Theorem 4.8\cite{FuTe/2002} it follows that the law of $(\tY,\tZ)$ is uniquely determined by the law of $(\btW, \tX)$
and by $\Phi$ and $\Psi$. We note that $\tY(t)$, being measurable with respect to the degenerate $\sigma$-algebra
$\tF_{0}$, is deterministic; in particular $\tY(t) = \bE(\tY(t))$ only depends on the law of $\tY$, and thus it is a
functional of $X_0,\bA, \bF, \sqrt{\bQ}, \Phi, \Psi$.
To stress dependence on the initial datum $X_0$, we will denote the solution of \eqref{eq:cont2} and \eqref{eq:bsde}
by $\{ (\tX^{X_0}(t), \tY^{X_0}(t), \tZ^{X_0}(t)), t \in [0, T]\}$.

We recall from [9, Theorem 6.2] that, Proposition \ref{prop:cont} and Hypothesis \ref{hp:psi}, imply
existence and uniqueness of a (mild) solution $u \in C^{0,1}([0, T] \times \cH; \R)$ of the Hamilton Jacobi Bellman equation corresponding with our control problem:
\begin{equation}
   \begin{cases}
  \frac{\partial}{\partial t} u(t, X) + \cL_t [u(t, \cdot)](X) \\
     \qquad \qquad = \psi(t, X, u(t, X),\nabla u(t, X)\sqrt{\bQ}), \quad  t\in [0, T], X \in  \cH,\\
u(T, X) = \phi(X).
   \end{cases}
\end{equation}
Here $\cL$ is the is the infinitesimal generator of the Markov semigroup corresponding to the process $X$:
\begin{align*}
    \cL_t [h](X)= \frac{1}{2}{\rm Tr}(\nabla^2 h(X)\bQ) + \langle \bA X + \bF(t,X), \nabla h(X) \rangle.
\end{align*}
Moreover, $\PP$-a.s. for a.e. $t \in [0, T]$, we have
\begin{align*}
\tY^{X_0}(t)=u(t,\tX^{X_0}(t)), \quad \qquad
\tZ^{X_0}(t)= \nabla u(t,\tX^{X_0}(t) )\sqrt{\bQ}.
\end{align*}
The relevance of the solution of the Hamilton-Jacobi-Bellman equation to our control problem is explained in the following
proposition.
\begin{prop}
  Assume that Hypotheses  \ref{hp:kernel},\ref{hp:f}, \ref{hp:f+cont}, \ref{hp:noise}, \ref{hp:rlphi} and \ref{hp:psi}  hold. For every $ t\in [0, T]$ and $X_0 \in \cH$, and
for every a.c.s. $\bU$ we have $u(0,X_0) \leq  \bJ(X_0,\bU)$ and equality holds if and only if the following feedback law is verified, $\PP$-a.s. for almost every 
$t \in [0, T]$:
\begin{align}\label{eq:Gamma1}
\hat{\gamma}(t) =\Gamma(t, \tX(t), \nabla u(t,\tX(t))\sqrt{\bQ}).
\end{align}
Finally, there
exists at least an a.c.s. $\bU$ verifying \eqref{eq:Gamma1}. 
In such a system, the closed loop equation admits a solution
\begin{equation}
    \begin{cases}
        \dd \bar{X}(t)= \bA \bar{X}(t) \dd t+ \bF(t,\bar{X}(t))\dd t + \\
        \qquad \quad \sqrt{Q} \left( R(t,\bar{X}(t),\Gamma(t,\bar{X}(t),\nabla u(t,\tX(t)) \sqrt{\bQ})\dd t + \dd \mathbf{W}(t) \right), 
         \ \ t\in [s,T]\\
        \bar{X}(s)=X_0 \in \cH,
    \end{cases}
\end{equation}
   and if $\bar{\gamma}(t)=\Gamma(t,\bar{X}(t),\nabla u(t,\tX(t) )\sqrt{\bQ})$
then the couple $(\bar{\gamma},\bar{X})$ is optimal for the control problem.
\end{prop}

\begin{proof}
    The result follows immediately from the paper of Fuhrman and Tessitore \cite[Theorem 7.2]{FuTe/2002}.
\end{proof}


\end{document}